\documentclass[12pt]{article}
\usepackage{amscd,amsmath,amssymb,amsfonts,braket,mathrsfs,latexsym,enumerate,amsthm}
\usepackage[all]{xy}
\makeatletter
\@namedef{subjclassname@2010}{%
  \textup{2010} Mathematics Subject Classification}
\makeatother

\newtheorem{thm}{Theorem} 
\newtheorem{lem}[thm]{Lemma} 

\newtheorem{prop}[thm]{Proposition}

\theoremstyle{definition}
\newtheorem{rmk}[thm]{Remark}

\numberwithin{thm}{section}
\numberwithin{equation}{section}

\newcommand{\C}{\mathbb{C}}
\newcommand{\F}{\mathbb{F}}

\newcommand{\Fbar}{\overline{F}}
\newcommand{\Q}{\mathbb{Q}}

\newcommand{\kb}{\overline{k}}

\newcommand{\Fsep}{F^{\text{sep}}}
\newcommand{\rhob}{\overline{\rho}}

\newcommand{\nubar}{\overline{\nu}}

\newcommand{\Fab}{F^{\text{ab}}}
\newcommand{\Gab}{\mathrm{G}^{\text{ab}}}
\newcommand{\lambdaab}{\lambda^{\text{ab}}}

\newcommand{\R}{\mathbb{R}}
\newcommand{\Z}{\mathbb{Z}}

\newcommand{\betab}{\overline{\beta}}

\newcommand{\mfa}{\mathfrak{a}}
\newcommand{\mfp}{\mathfrak{p}}
\newcommand{\mfq}{\mathfrak{q}}
\newcommand{\mfI}{\mathfrak{I}}

\newcommand{\Atil}{\widetilde{A}}

\newcommand{\cE}{\mathcal{E}}
\newcommand{\cP}{\mathcal{P}}

\newcommand{\cO}{\mathcal{O}}

\newcommand{\cFR}{\mathcal{FR}}

\newcommand{\cU}{\mathcal{U}}
\newcommand{\Ram}{\textsf{Ram}}

\newcommand{\cMn}{\mathcal{M}^{\textit{new}}}
\newcommand{\cNn}{\mathcal{N}^{\textit{new}}}
\newcommand{\cSn}{\mathcal{S}^{\textit{new}}}
\newcommand{\cTn}{\mathcal{T}^{\textit{new}}}

\newcommand{\id}{\mathrm{id}}

\newcommand{\Frob}{\mathrm{Frob}}

\newcommand{\Gal}{\mathrm{Gal}}
\newcommand{\G}{\mathrm{G}}
\newcommand{\M}{\mathrm{M}}
\newcommand{\N}{\mathrm{N}}
\newcommand{\GL}{\mathrm{GL}}

\newcommand{\End}{\mathrm{End}}
\newcommand{\Aut}{\mathrm{Aut}}

\newcommand{\im}{\mathrm{Im}\, }
\newcommand{\Norm}{\mathrm{Norm}}

\newcommand{\ord}{\mathrm{ord}}

\newcommand{\tr}{\mathrm{tr}}

\newcommand{\cf}{cf.\ }
\newcommand{\inj}{\hookrightarrow}

\newcommand{\resp}{resp.\ }

\newcommand{\ch}{\mathrm{char}\,}

\begin{document}

\title{Algebraic points on Shimura curves of $\Gamma_0(p)$-type (III)}
\author{Keisuke Arai}
\date{}
%\address[Keisuke Arai]
%{Department of Mathematics, School of Engineering,
%Tokyo Denki University,
%2-2 Kanda-Nishiki-cho, Chiyoda-ku, Tokyo, Japan 
%101-8457}
%\email{araik@mail.dendai.ac.jp}

%\address[Fumiyuki Momose]
%{Department of Mathematics, Faculty of Science and Engineering,
%Chuo University,
%1-13-27 Kasuga, Bunkyo-ku, Tokyo 112-8551, Japan}
%\email{momose@math.chuo-u.ac.jp}

%\date{}

%\pagestyle{plain}
%
%

%
%

%\subjclass[2010]{Primary 11G18, 14G05; Secondary 11G10, 11G15}

%\keywords{rational points, Shimura curves, QM-abelian surfaces}

\maketitle

%\begin{center}
%\textit{To the memory of Fumiyuki Momose}
%\end{center}

\begin{abstract}

In previous articles,
we classified the characters associated to algebraic points
on Shimura curves of $\Gamma_0(p)$-type, and over number fields
in a certain large class
we showed that
there are at most elliptic points on such a Shimura curve
for every sufficiently large prime number $p$.
%We also obtained
%an effective bound of $p$ concerning the classification of the characters
%and algebraic points on Shimura curves of $\Gamma_0(p)$-type.
%
%In this article, we slightly improve the estimate of $p$.
In this article, we prove the non-existence of elliptic points
on Shimura curves of $\Gamma_0(p)$-type under a mild assumption.
We also give an explicit example.

\end{abstract}

%\noindent
%2010 \textit{Mathematics Subject Classification.}
%Primary 11G18, 14G05; Secondary 11G10, 11G15.

%\tableofcontents

\section{Introduction}
\label{intro}

Let $B$ be an indefinite quaternion division algebra over $\Q$
of discriminant $d$.
Fix a maximal order $\cO$ of $B$.
%For each prime number $p$ not dividing $d$, fix an isomorphism
%\begin{equation}
%\label{OM2}
%\cO\otimes_{\Z}\Z_p\cong\M_2(\Z_p) \nonumber
%\leqno(1.1)
%\end{equation}
%of $\Z_p$-algebras.
%
%\begin{defn}
%\label{defqm}
%\rm
%
%(\cf \cite[p.591]{Bu})
%Let $S$ be a scheme.
%where $d$ is invertible.
A \textit{QM-abelian surface by $\cO$} over a scheme $S$ is a pair $(A,i)$ where
$A$ is an abelian scheme over $S$ of relative dimension $2$, and 
$i:\cO\inj\End_S(A)$ 
is an injective ring homomorphism (sending $1$ to $\id$)
(\cf \cite[p.591]{Bu}).
Here $\End_S(A)$ is the ring of endomorphisms of $A$ defined over $S$.
We assume that $A$ has a left $\cO$-action.
We will sometimes omit ``by $\cO$" and simply write ``a QM-abelian surface"
if there is no fear of confusion.

%\end{defn}

%\noindent
Let $M^B$ be the coarse moduli scheme over $\Q$ parameterizing isomorphism classes
of QM-abelian surfaces by $\cO$ (\cf \cite[p.93]{J}).
Then $M^B$ is a proper smooth curve over $\Q$, called a \textit{Shimura curve}.
For a prime number $p$ not dividing $d$,
let $M_0^B(p)$
% (or simply $M_0(p)$) 
be the coarse moduli scheme over $\Q$ parameterizing isomorphism classes
of triples $(A,i,V)$ where $(A,i)$ is a QM-abelian surface by $\cO$
and $V$ is a left $\cO$-submodule of $A[p]$ with $\F_p$-dimension $2$.
Here $A[p]$ is the kernel of multiplication by $p$ in $A$.
Then $M_0^B(p)$ is a proper smooth curve over $\Q$, which we call a
\textit{Shimura curve of $\Gamma_0(p)$-type}.
We have a natural map $$\pi^B(p):M_0^B(p)\longrightarrow M^B$$
over $\Q$ defined by $(A,i,V)\longmapsto (A,i)$.
Note that $M^B$ (\resp $M_0^B(p)$) is an analogue of the modular curve $X_0(1)$
(\resp $X_0(p)$).

%We study points on $M_0^B(p)$.
We say that a prime of a number field is \textit{of odd degree}
if the cardinality of the residue field is an odd power of the
residual characteristic.
The main result of this article is:

\begin{thm}
\label{mainthm}

Let $k$ be a finite Galois extension of $\Q$ which does not contain
the Hilbert class field of any imaginary quadratic field.
Assume that there is a prime $\mfq$ of $k$ such that
$\mfq$ is of odd degree,
the residual characteristic $q$ of $\mfq$ is unramified in $k$,
and
$B\otimes_{\Q}\Q(\sqrt{-q})\not\cong\M_2(\Q(\sqrt{-q}))$.
Then $M_0^B(p)(k)=\emptyset$ holds for every sufficiently large prime number $p$.

\end{thm}

\begin{rmk}
\label{ellippt?}

\begin{enumerate}[\upshape (1)]
\setlength{\itemsep}{0mm}
\setlength{\parskip}{0mm}
\item
Theorem \ref{mainthm} is an analogue of the results for points on the modular curve $X_0(p)$
(\cf \cite{Ma}, \cite{Mo}).
\item
We see that the Riemann surface $M_0^B(p)(\C)$ is isomorphic to a quotient of the upper half-plane,
and it often has elliptic points of order $2$ or $3$ (\cf \cite[\S 3]{AM}).
\item
In previous articles \cite{A4}, \cite{A5}, \cite{AM} (\cf \cite{AM2}), we could not exclude the
possibility of the existence of elliptic points in $M_0^B(p)(k)$.
\item
We see $M^B(\R)=\emptyset$ by \cite[Theorem 0]{Sh}.
Since there is a map
$\pi^B(p):M_0^B(p)\longrightarrow M^B$
over $\Q$, we have $M_0^B(p)(\R)=\emptyset$
for any prime number $p$ (not dividing $d$).
\end{enumerate}

\end{rmk}

%The author is very sorry for the loss Fumiyuki Momose,
%and dedicates this work to his memory.

%
\vspace{5mm}
\noindent
{\bf Notation}

For a field $F$,
let $\ch F$ denote the characteristic of $F$,
let $\Fbar$ denote an algebraic closure of $F$,
let $\Fsep$ (\resp $\Fab$) denote the separable closure
(\resp the maximal abelian extension) of $F$ inside $\Fbar$,
and let $\G_F=\Gal(\Fsep/F)$, $\Gab_F=\Gal(\Fab/F)$.
For a prime number $p$ and a field $F$ with $\ch F\ne p$, let
$\theta_p:\G_F\longrightarrow\F_p^{\times}$ denote the mod $p$ cyclotomic character.

For a number field $k$,
let $\cO_k$ denote the ring of integers of $k$;
put $\N(\mfq):=\sharp(\cO_k/\mfq)$ for a prime $\mfq$ of $k$;
let $Cl_k$ denote the ideal class group of $k$;
let $h_k$ denote the class number of $k$;
fix an inclusion $k\hookrightarrow\C$
and take the algebraic closure $\kb$ inside $\C$;
let $k_v$ denote the completion of $k$ at $v$
where $v$ is a place (or a prime) of $k$;
%let $k_{\A}$ denote the ad\`{e}le ring of $k$;
and let $\Ram (k)$ denote the set of prime numbers which are ramified in $k$.

\section{Galois representations associated to QM-abelian surfaces (generalities)}
\label{QM}

We review \cite[\S 3]{AM} briefly in order to
consider the Galois representations associated to a QM-abelian surface.
Take a prime number $p$ not dividing $d$.
Let $F$ be a field with $\ch F\ne p$.
Let $(A,i)$ be a QM-abelian surface by $\cO$ over $F$.
The action of $\G_F$ on $A[p](\Fsep)\cong\F_p^4$
determines a representation
$\rhob:\G_F\longrightarrow\GL_4(\F_p)$.
By a suitable choice of basis, $\rhob$ factors as
$$\rhob:\G_F\longrightarrow
\left\{\left(\begin{matrix} sI_2&tI_2 \\ uI_2&vI_2\end{matrix}\right) \Bigg|\,
\left(\begin{matrix} s&t \\ u&v\end{matrix}\right)\in\GL_2(\F_p)\right\}
\subseteq\GL_4(\F_p),$$
where
$I_2=\left(\begin{matrix} 1&0 \\ 0&1\end{matrix}\right)$.
Let
\begin{equation}
\label{rhobar}
\rhob_{A,p}:\G_F\longrightarrow\GL_2(\F_p)
\end{equation}
denote the Galois representation induced from $\rhob$ by
$``\left(\begin{matrix} s&t \\ u&v\end{matrix}\right)"$,
so that we have
$\rhob_{A,p}(\sigma)=\left(\begin{matrix}
s(\sigma)&t(\sigma) \\ u(\sigma)&v(\sigma)\end{matrix}\right)$
if
$\rhob(\sigma)=\left(\begin{matrix}
s(\sigma)I_2&t(\sigma)I_2 \\ u(\sigma)I_2&v(\sigma)I_2\end{matrix}\right)$
for $\sigma\in\G_F$.

Suppose that $A[p](\Fsep)$ has a left $\cO$-submodule $V$ with $\F_p$-dimension $2$
which is stable under the action of $\G_F$.
Then, by taking a conjugate if necessary, we may assume
$\rhob_{A,p}(\G_F)\subseteq
\left\{\left(\begin{matrix} 
s&t \\ 0&v
\end{matrix}\right)\right\}\subseteq\GL_2(\F_p)$.
Let
\begin{equation}
\label{lambda}
\lambda:\G_F\longrightarrow\F_p^{\times}
\end{equation}
denote the character induced from $\rhob_{A,p}$ by ``$s$", so that
$\rhob_{A,p}(\sigma)=
\left(\begin{matrix} 
\lambda(\sigma)&* \\ 0&*
\end{matrix}\right)$
for $\sigma\in\G_F$.
Note that
$\G_F$ acts on $V$ by $\lambda$
(i.e. $\rhob(\sigma)(v)=\lambda(\sigma)v$
for $\sigma\in\G_F$, $v\in V$).

\section{Automorphism groups}
\label{Aut}

We give a brief summary of \cite[\S 4]{AM} 
concerning the automorphism groups of a QM-abelian surface.
Let $(A,i)$ be a QM-abelian surface by $\cO$ over a field $F$.
Let $\End(A)$ (\resp $\Aut(A)$) denote the ring of endomorphisms
(the group of automorphisms) of $A$ defined over $\Fbar$.
Put
$$\End_{\cO}(A):=\{f\in\End(A)\mid f\circ i(g)=i(g)\circ f \text{\ \ for any $g\in\cO$}\},$$
$$\Aut_{\cO}(A):=\Aut(A)\cap\End_{\cO}(A).$$
If $\ch F=0$,
then $\Aut_{\cO}(A)\cong\Z/2\Z$, $\Z/4\Z$ or $\Z/6\Z$.
%The isomorphism $\Aut_{\cO}(A)\cong\Z/4\Z$ (\resp $\Aut_{\cO}(A)\cong\Z/6\Z$)
%is possible only when all the prime divisors of $d$ are congruent to $2$ or $-1$ modulo $4$
%(\resp $0$ or $-1$ modulo $3$).

Let $p$ be a prime number not dividing $d$.
Let $(A,i,V)$ be a triple where $(A,i)$ is a QM-abelian surface by $\cO$ over
a field $F$
and $V$ is a left $\cO$-submodule of $A[p](\Fbar)$ with $\F_p$-dimension $2$.
Define a subgroup $\Aut_{\cO}(A,V)$ of $\Aut_{\cO}(A)$ by
$$\Aut_{\cO}(A,V):=\{f\in\Aut_{\cO}(A)\mid f(V)=V\}.$$
Assume $\ch F=0$.
Then $\Aut_{\cO}(A,V)\cong\Z/2\Z$, $\Z/4\Z$ or $\Z/6\Z$.
Note that we have
$\Aut_{\cO}(A)\cong\Z/2\Z$
(\resp $\Aut_{\cO}(A,V)\cong\Z/2\Z$)
if and only if
$\Aut_{\cO}(A)=\{\pm 1\}$
(\resp $\Aut_{\cO}(A,V)=\{\pm 1\}$).

\section{Fields of definition}
\label{fieldofdefinition}

We review \cite[\S 4]{AM} to consider the field of definition of a point
on $M_0^B(p)$.
Let $k$ be a number field.
Let $p$ be a prime number not dividing $d$.
Take a point
$$x\in M_0^B(p)(k).$$
Let $x'\in M^B(k)$ be the image
of $x$ by the map $\pi^B(p):M_0^B(p)\longrightarrow M^B$.
Then $x'$ is represented by a QM-abelian surface (say $(A_x,i_x)$) over $\kb$,
and $x$ is represented by a triple $(A_x,i_x,V_x)$
where $V_x$ is a left $\cO$-submodule of $A[p](\kb)$
with $\F_p$-dimension $2$.
For a finite extension $M$ of $k$,
we say that
\textit{we can take $(A_x,i_x,V_x)$ to be defined over $M$}
if there is a QM-abelian surface $(A,i)$ over $M$
and a left $\cO$-submodule $V$ of $A[p](\kb)$ with $\F_p$-dimension $2$
stable under the action of $\G_M$ such that there is an isomorphism
between $(A,i)\otimes_M\kb$ and $(A_x,i_x)$ under which $V$ corresponds
to $V_x$.
Put
$$\Aut(x):=\Aut_{\cO}(A_x,V_x),\ \ \ \ \Aut(x'):=\Aut_{\cO}(A_x).$$
Then $\Aut(x)$ is a subgroup of $\Aut(x')$.
Note that $x$ is an elliptic point of order $2$ (\resp $3$)
if and only if
$\Aut(x)\cong\Z/4\Z$
(\resp $\Aut(x)\cong\Z/6\Z$).
Since $x$ is a $k$-rational point, we have 
$^{\sigma}x=x$ for any $\sigma\in\G_k$.
Then, for any $\sigma\in\G_k$, there is an isomorphism
$$\phi_{\sigma}:{}^{\sigma}(A_x,i_x,V_x)\longrightarrow (A_x,i_x,V_x),$$
which we fix once for all.
For $\sigma,\tau\in\G_k$, put
$$c_x(\sigma,\tau)
:=\phi_{\sigma}\circ {}^{\sigma}\phi_{\tau}
\circ\phi_{\sigma\tau}^{-1}\in\Aut(x).$$
Then $c_x$ is a $2$-cocycle
and defines a cohomology class
$[c_x]\in H^2(\G_k,\Aut(x))$.
Here the action of $\G_k$ on $\Aut(x)$
is defined in a natural manner (\cf \cite[\S 4]{AM}).
For a place $v$ of $k$, let $[c_x]_v\in H^2(\G_{k_v},\Aut(x))$
denote the restriction of $[c_x]$ to $\G_{k_v}$.

\begin{prop}[{\cite[Proposition 4.2]{AM}}]
\label{fieldM0Bp}

\begin{enumerate}[\upshape (1)]
\setlength{\itemsep}{0mm}
\setlength{\parskip}{0mm}
\item
Suppose $B\otimes_{\Q}k\cong\M_2(k)$.
Further assume $\Aut(x)\ne\{\pm 1\}$ or
$\Aut(x')\not\cong\Z/4\Z$.
Then we can take $(A_x,i_x,V_x)$ to be defined over $k$.
\item
Assume $\Aut(x)=\{\pm 1\}$.
Then there is a quadratic extension $K$ of $k$
such that we can take $(A_x,i_x,V_x)$ to be defined over $K$.
\end{enumerate}
\end{prop}

\begin{lem}[{\cite[Lemma 4.3]{AM}}]
\label{fieldofdef}

Let $K$ be a quadratic extension of $k$.
Assume $\Aut(x)=\{\pm 1\}$.
Then the following two conditions are equivalent:
%
%\begin{enumerate}[\upshape (1)]
%\setlength{\itemsep}{0mm}
%\setlength{\parskip}{0mm}
%\item

\noindent
{\rm (1)}
We can take $(A_x,i_x,V_x)$ to be defined over $K$.
%
%\item

\noindent
{\rm (2)}
For any place $v$ of $k$ satisfying $[c_x]_v\ne 0$,
the tensor product $K\otimes_k k_v$ is a field.
%\end{enumerate}
\end{lem}

\section{Classification of characters}
\label{char}

We keep the notation in Section \ref{fieldofdefinition}.
Throughout this section, assume $\Aut(x)=\{\pm 1\}$.
Let $K$ be a quadratic extension of $k$
which satisfies the equivalent conditions in
Lemma \ref{fieldofdef}.
Then $x$ is represented by a triple $(A,i,V)$,
where $(A,i)$ is a QM-abelian surface over $K$ and $V$ is a left
$\cO$-submodule of $A[p](\kb)$ with $\F_p$-dimension $2$
stable under the action of $\G_K$.
Let
$$\lambda:\G_K\longrightarrow\F_p^{\times}$$
be the character associated to $V$ in (\ref{lambda}).
Let $\lambdaab:\Gab_K\longrightarrow\F_p^{\times}$
be the natural map induced from $\lambda$.
Put
\begin{equation}
\label{phi}
\varphi:=\lambdaab\circ\tr_{K/k}:\G_k\longrightarrow\Gab_K
\longrightarrow\F_p^{\times},
\end{equation}
where $\tr_{K/k}:\G_k\longrightarrow\Gab_K$ is the transfer map.
By \cite[Lemma 5.1]{AM} (\resp \cite[Corollary 5.2]{AM}), the character
$\lambda^{12}$ (\resp $\varphi^{12}$) is unramified at every prime of $K$
(\resp $k$) not dividing $p$,
and so it
corresponds to a character of the
ideal group $\mfI_K(p)$ (\resp $\mfI_k(p)$) consisting of
fractional ideals of $K$ (\resp $k$) prime to $p$.
By abuse of notation, let $\lambda^{12}$ (\resp $\varphi^{12}$) also denote
the corresponding character
of $\mfI_K(p)$ (\resp $\mfI_k(p)$).

Let $\cMn$ be the set of prime numbers which split
completely in $k$.
%
%Note that no prime number $q$ in the set $\cM$ in \cite[\S 5]{AM}
%was supposed to divide $6h_k$.
Let $\cNn$ be the set of primes of $k$
which divide some prime number $q\in \cMn$.
Take a finite subset $\emptyset\ne \cSn\subseteq \cNn$
which generates $Cl_k$.
For each prime $\mfq\in \cSn$, fix an element $\alpha_{\mfq}\in\cO_k\setminus\{0\}$
satisfying $\mfq^{h_k}=\alpha_{\mfq}\cO_k$.
For an integer $n\geq 1$, put
$$\cFR(n):=\Set{\beta\in\C|
\beta^2+a\beta+n=0 \text{ for some integer $a\in\Z$ with $|a|\leq 2\sqrt{n}$}}.$$
For any element $\beta\in\cFR(n)$, we have $|\beta|=\sqrt{n}$.
When $k$ is Galois over $\Q$, define the sets

\noindent
$\cE(k):=\Set{\varepsilon_0=\displaystyle\sum_{\sigma\in\Gal(k/\Q)}a_{\sigma}\sigma
\in\Z[\Gal(k/\Q)]|
a_{\sigma}\in\{0,8,12,16,24 \}}$,

\noindent
$\cMn_1(k):=
\Set{(\mfq,\varepsilon_0,\beta_{\mfq})|
\mfq\in \cSn,\ \varepsilon_0
\in\cE(k),\ 
\beta_{\mfq}\in\cFR(\N(\mfq))}$,

\noindent
$\cMn_2(k):=\Set{\Norm_{k(\beta_{\mfq})/\Q}(\alpha_{\mfq}^{\varepsilon_0}-\beta_{\mfq}^{24h_k})\in\Z|
(\mfq,\varepsilon_0,\beta_{\mfq})\in\cMn_1(k)}\setminus\{0\}$,

\noindent
$\cNn_0(k):=\Set{\text{$l$ : prime number}|\text{$l$ divides some integer $m\in\cMn_2(k)$}}$,

\noindent
$\cTn(k):=\Set{\text{$l'$ : prime number}|\text{$l'$ is divisible
by some prime $\mfq'\in \cSn$}}
\cup\{2,3\}$,

\noindent
$\cNn_1(k):=\cNn_0(k)\cup\cTn(k)\cup\Ram(k)$.

\noindent
Note that all the sets, $\cFR(n)$, $\cE(k)$, $\cMn_1(k)$, $\cMn_2(k)$, $\cNn_0(k)$, $\cTn(k)$,
and $\cNn_1(k)$, are finite.
In \cite{A5}, an upper bound of $\cNn_1(k)$ is given.
We have the following classification of $\varphi$:

\begin{thm}[{\cite[Theorem 5.1]{A5}}]
\label{type23phi}

Assume that $k$ is Galois over $\Q$.
If $p\not\in\cNn_1(k)$
(and if $p$ does not divide $d$),
then the character
$\varphi:\G_k\longrightarrow \F_p^{\times}$
is of one of the following types:
\begin{description}
\setlength{\itemsep}{0mm}
\setlength{\parskip}{0mm}
\item
{\rm Type 2}.
$\varphi^{12}=\theta_p^{12}$ and $p\equiv 3\bmod{4}$.
\item
{\rm Type 3}.
There is an imaginary quadratic field $L$ satisfying the following
conditions:
\begin{enumerate}[\upshape (a)]
\setlength{\itemsep}{0mm}
\setlength{\parskip}{0mm}
\item
The Hilbert class field $H_L$ of $L$ is contained in $k$.
\item
There is a prime $\mfp_L$ of $L$ lying over $p$
such that
$\varphi^{12}(\mfa)\equiv\delta^{24}\bmod{\mfp_L}$ holds
for any fractional ideal $\mfa$ of $k$ prime to $p$.
Here $\delta$ is any element of $L$ such that
$\Norm_{k/L}(\mfa)=\delta\cO_L$.
\end{enumerate}
\end{description}
\end{thm}

%From now to the end of this section, assume that $k$ is Galois over $\Q$.

\begin{lem}[{\cite[Lemma 5.2]{A5}}]
\label{type2lambda}

Suppose that $k$ is Galois over $\Q$, and $p\geq 11$, $p\ne 13$, $p\not\in\cNn_1(k)$.
Further assume the following two conditions:
\begin{enumerate}[\upshape (a)]
\setlength{\itemsep}{0mm}
\setlength{\parskip}{0mm}
\item
Every prime of $k$ above $p$ is inert in $K/k$.
\item
Every prime of $k$ in $\cSn$ is ramified in $K/k$.
\end{enumerate}
\noindent
If $\varphi$ is of type 2, then we have the following assertions:
\begin{enumerate}[\upshape (i)]
\setlength{\itemsep}{0mm}
\setlength{\parskip}{0mm}
\item
The character $\lambda^{12}\theta_p^{-6}:\G_K\longrightarrow\F_p^{\times}$
is unramified everywhere.
\item
The map $Cl_K\longrightarrow\F_p^{\times}$ induced from $\lambda^{12}\theta_p^{-6}$
is trivial on
$C_{K/k}:=\im(Cl_k\longrightarrow Cl_K)$,
where
$Cl_k\longrightarrow Cl_K$
is the map defined by
$[\mfa]\longmapsto[\mfa\cO_K]$.
\end{enumerate}
\end{lem}

From now to the end of this section,
suppose that $k$ is Galois over $\Q$, $p\geq 11$, $p\ne 13$, $p\not\in\cNn_1(k)$,
and that $\varphi$ is of type 2.
Let $q\ne p$ be a prime number, and
take a prime $\mfq$ of $k$ above $q$.
By replacing $K$ if necessary,
we may assume the conditions (a), (b) in Lemma \ref{type2lambda}
and that $\mfq$ is ramified in $K/k$ (\cf \cite[Remark 4.4]{AM}).
Let $\mfq_K$ be the unique prime of $K$ above $\mfq$.
The abelian surface $A\otimes_K K_{\mfq_K}$ has good reduction
after a totally ramified finite extension $M/K_{\mfq_K}$ (\cf \cite[Proposition 3.2]{J}).
Let $\Atil$ be the special fiber of the N\'{e}ron model of
$A\otimes_K M$.
Then $\Atil$ is a QM-abelian surface by $\cO$ over
$\cO_k/\mfq$.
We have $\lambda(\Frob_M)\equiv\beta$
modulo a prime $\mfp_0$ of $\Q(\beta)$ above $p$
for a Frobenius eigenvalue
$\beta$ of $\Atil$,
where $\Frob_M$ is the arithmetic Frobenius of $\G_M$
($\subseteq\G_{K_{\mfq_K}}$).
We see $\beta\in\cFR(\N(\mfq))$ by \cite[p.97]{J}.
Since $\det\rhob_{A,p}=\theta_p$ (\cf \cite[Proposition 1.1 (2)]{Oh}),
we have $\lambda^{-1}\theta_p(\Frob_M)\equiv\betab\bmod{\mfp_0}$,
where $\betab$ is the complex conjugate of $\beta$.
Put $$\psi:=\lambda\theta_p^{-\frac{p+1}{4}}.$$
Then
$\psi^{12}=\lambda^{12}\theta_p^{-3(p+1)}
=\lambda^{12}\theta_p^{-6}$.
%
%Assume that $\varphi$ is of type 2.
%

\begin{lem}
\label{b^2+bb^2}

\begin{enumerate}[\upshape (1)]
\setlength{\itemsep}{0mm}
\setlength{\parskip}{0mm}
\item
$\psi(\Frob_M)^6=1$.
\item
$\psi(\Frob_M)^2+\psi(\Frob_M)^{-2}=-1$ or $2$.
\item
$\beta^2+\betab^2\equiv-\N(\mfq)^{\frac{p+1}{2}}$ or
$2\N(\mfq)^{\frac{p+1}{2}}\bmod{p}$.
\end{enumerate}

\end{lem}

\begin{proof}

(1)
By Lemma \ref{type2lambda} (ii),
we have
$1=\lambda^{12}(\mfq\cO_K)\theta_p^{-6}(\mfq\cO_K)
=\psi^{12}(\mfq\cO_K)=\psi^{24}(\mfq_K)
=\psi^{24}(\Frob_M)=\psi(\Frob_M)^{24}$.
%
%Here, note that $\psi(\Frob_M)$ is well-defined
%since $A\otimes_K M$ has good reduction.
Note that the fourth equality holds because the extension
$M/K_{\mfq_K}$ is totally ramified.
Since $\F_p^{\times}$ is a cyclic group of order $p-1$ and
$p-1\equiv 2\bmod{4}$, we obtain $\psi(\Frob_M)^6=1$.

(2)
This follows immediately from (1).

(3)
%\begin{multline*}
$\beta^2+\betab^2\equiv
\psi(\Frob_M)^2\theta_p(\Frob_M)^{\frac{p+1}{2}}
+\psi(\Frob_M)^{-2}\theta_p(\Frob_M)^{\frac{-p+3}{2}}$

\noindent
$=\theta_p(\Frob_M)^{\frac{p+1}{2}}
(\psi(\Frob_M)^2+\psi(\Frob_M)^{-2})
=-\N(\mfq)^{\frac{p+1}{2}} \text{\ or\ } 2\N(\mfq)^{\frac{p+1}{2}}\bmod{p}$.
%\end{multline*}

\end{proof}

We repeat the argument in \cite[\S 5]{A4}
when $\mfq$ is of odd degree as follows:

\begin{lem}
\label{BQ(-q)M2}
%\label{q/p-1}

%Suppose $p\geq 11$, $p\ne 13$ and $p\not\in\cNn_1(k)$.
%Assume that $\varphi$ is of type 2.
%Further
Suppose that $\mfq$ is of odd degree and satisfies $\N(\mfq)<\frac{p}{4}$.
Then:
%$\left(\frac{\N(\mfq)}{p}\right)=-1$
%and

\noindent
{\rm (1)}
$\N(\mfq)^{\frac{p-1}{2}}\equiv -1\bmod{p}$.

\noindent
{\rm (2)}
$(\beta+\betab)^2\equiv 3\N(\mfq)$ or $0\bmod{p}$.

\noindent
{\rm (3)}
$q=3$ and $|\beta+\betab|=\sqrt{3\N(\mfq)}$, or
$\beta+\betab=0$.

\noindent
{\rm (4)}
$B\otimes_{\Q}\Q(\sqrt{-q})\cong\M_2(\Q(\sqrt{-q}))$.

\end{lem}

\begin{proof}

(1)
Assume otherwise i.e. 
%$\left(\frac{\N(\mfq)}{p}\right)=1$.
%Then
$\N(\mfq)^{\frac{p-1}{2}}\equiv 1\bmod{p}$.
%
%Since $\N(\mfq)^{\frac{p-1}{2}}\equiv 1\bmod{p}$, we have
%$\N(\mfq)^{\frac{p+1}{2}}\equiv \N(\mfq)\bmod{p}$.
Then Lemma \ref{b^2+bb^2} (3) implies $\beta^2+\betab^2\equiv -\N(\mfq)$ or $2\N(\mfq)\bmod{p}$,
and so $(\beta+\betab)^2\equiv \N(\mfq)$ or $4\N(\mfq)\bmod{p}$.
Since $\beta\in\cFR(\N(\mfq))$, we have
$|\beta+\betab|\leq 2\sqrt{\N(\mfq)}$.
Then
$|(\beta+\betab)^2-\N(\mfq)|\leq 3\N(\mfq)<p$ and
$|(\beta+\betab)^2-4\N(\mfq)|\leq 4\N(\mfq)<p$.
Hence
$(\beta+\betab)^2=\N(\mfq)$ or $4\N(\mfq)$.
Since \textit{$\mfq$ is of odd degree}, this contradicts $\beta+\betab\in\Z$.
Therefore we conclude
%$\left(\frac{\N(\mfq)}{p}\right)=-1$, and so
$\N(\mfq)^{\frac{p-1}{2}}\equiv -1\bmod{p}$.

(2)
By (1) and Lemma \ref{b^2+bb^2} (3),
we have $\beta^2+\betab^2\equiv \N(\mfq)$ or $-2\N(\mfq)\bmod{p}$.
Therefore $(\beta+\betab)^2\equiv 3\N(\mfq)$ or $0\bmod{p}$.

(3)
We have $(\beta+\betab)^2\leq 4\N(\mfq)$.
%since $\beta$ is an element of $\cFR(\N(\mfq))$.
%
%In fact we have $(\beta+\betab)^2< 4\N(\mfq)$
%because $\mfq$ is of odd degree,
%but we don't use this inequality here.
%
First assume 
$(\beta+\betab)^2\equiv 3\N(\mfq)\bmod{p}$.
Then, since $|(\beta+\betab)^2-3\N(\mfq)|\leq 3\N(\mfq)<p$,
we have $(\beta+\betab)^2=3\N(\mfq)$.
Therefore $q=3$ and $|\beta+\betab|=\sqrt{3\N(\mfq)}$.
Next assume
$(\beta+\betab)^2\equiv 0\bmod{p}$.
Then, since $|(\beta+\betab)^2|\leq 4\N(\mfq)<p$,
we have $(\beta+\betab)^2=0$.
Therefore $\beta+\betab=0$.

(4)
The number $\beta$ is a Frobenius eigenvalue of the QM-abelian surface $\Atil$
by $\cO$ over $\cO_k/\mfq$, where $\mfq$ is of odd degree.
Then, by (3) and
\cite[Theorem 2.1 (2) (4) and Proposition 2.3]{J},
we conclude
$\End_{\cO_k/\mfq}(\Atil)\otimes_{\Z}\Q
\cong\M_2(\Q(\sqrt{-q}))
\cong B\otimes_{\Q}\Q(\sqrt{-q})$.
Here $\End_{\cO_k/\mfq}(\Atil)$ is the ring of endomorphisms of $\Atil$
defined over $\cO_k/\mfq$.

\end{proof}

\section{Irreducibility of $\rhob_{A,p}$ and algebraic points on $M_0^B(p)$}
\label{sec:irred&alg}

Let $k$ be a number field, and
let $(A,i)$ be a QM-abelian surface by $\cO$ over $k$.
For a prime number $p$ not dividing $d$,
assume that the representation $\rhob_{A,p}$ in (\ref{rhobar}) is reducible.
Then there is a 1-dimensional sub-representation of $\rhob_{A,p}$,
and let
$$\nu:\G_k\longrightarrow\F_p^{\times}$$
be its associated character.
In this case there is a left $\cO$-submodule $V$
of $A[p](\kb)$ with $\F_p$-dimension $2$ on which $\G_k$ acts by $\nu$,
and so the triple $(A,i,V)$ determines a point $x\in M_0^B(p)(k)$.
Take any quadratic extension $K$ of $k$.
Then we have the characters
$\lambda:\G_K\longrightarrow\F_p^{\times}$
and
$\varphi:\G_k\longrightarrow\F_p^{\times}$
associated to the triple $(A\otimes_k K,i,V)$.
Note that we have $\varphi=\nu^2$ by construction of $\varphi$.
The following theorems generalize \cite[Theorems 6.1 and 6.2]{A5} slightly:

\begin{thm}
\label{irred}

Let $k$ be a finite Galois extension of $\Q$ which does not contain
the Hilbert class field of any imaginary quadratic field.
Assume that there is a prime $\mfq$ of $k$ such that $\mfq$ is of odd degree
and the residual characteristic $q$ of $\mfq$
satisfies $B\otimes_{\Q}\Q(\sqrt{-q})\not\cong\M_2(\Q(\sqrt{-q}))$.
Let $p>4\N(\mfq)$ be a prime number which also satisfies $p\geq 11$, $p\ne 13$,
$p\nmid d$ and $p\not\in\cNn_1(k)$.
Then the representation
$$\rhob_{A,p}:\G_k\longrightarrow\GL_2(\F_p)$$
is irreducible.

\end{thm}

\begin{proof}

Assume that $\rhob_{A,p}$ is reducible.
Then the associated character $\varphi$ is of type 2 in Theorem \ref{type23phi},
because $k$ does not contain
the Hilbert class field of any imaginary quadratic field.
By Lemma \ref{BQ(-q)M2} (4) we have
$B\otimes_{\Q}\Q(\sqrt{-q})\cong\M_2(\Q(\sqrt{-q}))$,
which is a contradiction.

\end{proof}

\begin{thm}
\label{mainthm0}

Let $k$ be a finite Galois extension of $\Q$ which does not contain
the Hilbert class field of any imaginary quadratic field.
Assume that there is a prime $\mfq$ of $k$ such that
$\mfq$ is of odd degree and the residual characteristic $q$ of $\mfq$ satisfies
$B\otimes_{\Q}\Q(\sqrt{-q})\not\cong\M_2(\Q(\sqrt{-q}))$.
Let $p>4\N(\mfq)$ be a prime number which also satisfies
$p\geq 11$, $p\ne 13$, $p\nmid d$ and $p\not\in\cNn_1(k)$.

\noindent
{\rm (1)}
Suppose $B\otimes_{\Q}k\cong\M_2(k)$. Then $M_0^B(p)(k)=\emptyset$.

\noindent
{\rm (2)}
Suppose $B\otimes_{\Q}k\not\cong\M_2(k)$. Then
$M_0^B(p)(k)\subseteq\{\text{elliptic points of order $2$ or $3$}\}$.

\end{thm}

\begin{proof}

%Let $k$ be a finite Galois extension of $\Q$ which does not contain
%the Hilbert class field of any imaginary quadratic field,
%and let $\mfq$ be a prime of $k$ such that 
%$\mfq$ is of odd degree
%and the residual characteristic $q$ of $\mfq$
%satisfies $B\otimes_{\Q}\Q(\sqrt{-q})\not\cong\M_2(\Q(\sqrt{-q}))$.
%Let $p>4\N(\mfq)$ be a prime number which also satisfies $p\geq 11$, $p\ne 13$,
%$p\nmid d$ and $p\not\in\cNn_1(k)$.
%
Take a point $x\in M_0^B(p)(k)$.

%(1)
%Suppose $B\otimes_{\Q}k\cong\M_2(k)$.

(1)
(1-i)
Assume $\Aut(x)\ne\{\pm 1\}$
or $\Aut(x')\not\cong\Z/4\Z$.
Then $x$ is represented by a triple $(A,i,V)$ defined over $k$
by Proposition \ref{fieldM0Bp} (1),
and the representation $\rhob_{A,p}$ is reducible.
This contradicts Theorem \ref{irred}.

(1-ii)
Assume otherwise (i.e. $\Aut(x)=\{\pm 1\}$
and $\Aut(x')\cong\Z/4\Z$).
Then $x$ is represented by a triple $(A,i,V)$ defined over a quadratic extension of $k$
by Proposition \ref{fieldM0Bp} (2),
and we have a character
$\varphi:\G_k\longrightarrow\F_p^{\times}$ as in (\ref{phi}).
By Theorem \ref{type23phi} and Lemma \ref{BQ(-q)M2} (4),
we have $B\otimes_{\Q}\Q(\sqrt{-q})\cong\M_2(\Q(\sqrt{-q}))$.
This is a contradiction.

(2)
%Suppose $B\otimes_{\Q}k\not\cong\M_2(k)$.
Assume that $x$ is not an elliptic point of order $2$ or $3$.
Then $\Aut(x)=\{\pm 1\}$.
By the same argument as in (1-ii),
we have a contradiction.

\end{proof}

\section{Elimination of elliptic points}
\label{sec:elimination}

In this section, we deduce Theorem \ref{mainthm} from Theorem \ref{mainthm0}.

\begin{prop}
\label{W&q'}

Let $k$ be a finite Galois extension of $\Q$ which does not contain
the Hilbert class field of any imaginary quadratic field.
Assume that there is a prime $\mfq$ of $k$ such that
$\mfq$ is of odd degree,
the residual characteristic $q$ of $\mfq$ is unramified in $k$,
and
$B\otimes_{\Q}\Q(\sqrt{-q})\not\cong\M_2(\Q(\sqrt{-q}))$.
%
%Further assume $B\otimes_{\Q}k\not\cong\M_2(k)$.
%
Then there is a finite Galois extension $W$ of $\Q$ satisfying
the following conditions:
\begin{enumerate}[\upshape (i)]
\setlength{\itemsep}{0mm}
\setlength{\parskip}{0mm}
\item
The composite field $kW$ does not contain
the Hilbert class field of any imaginary quadratic field.
\item
There is a prime $\mfq'$ of $kW$ of odd degree with residual characteristic $q$.
\item
$B\otimes_{\Q}kW\cong\M_2(kW)$.
\end{enumerate}
\end{prop}

\begin{thm}
\label{mainthm1}

In the situation of Proposition \ref{W&q'}, let $p$ be a prime number
satisfying
$p>4\N(\mfq')$, $p\geq 11$, $p\ne 13$, $p\nmid d$, and $p\not\in\cNn_1(kW)$.
Then $M_0^B(p)(k)=M_0^B(p)(kW)=\emptyset$.

\end{thm}

\begin{proof}

Applying Theorem \ref{mainthm0} (1) to $kW$, we obtain the result.

\end{proof}

Theorem \ref{mainthm} follows immediately from Theorem \ref{mainthm1}.
From now to the end of this section, suppose that the assumption in
Proposition \ref{W&q'} holds.
Fix a prime $\mfq$ of $k$ as in Proposition \ref{W&q'}.
Let $\cU$ be the set of non-zero integers $N\in\Z$ such that
\begin{itemize}
\item
$N$ is square free,
\item
$d\mid N$,
\item
$q\mid N$.
\end{itemize}
For an integer $N\in\Z$, put
$W_N:=\Q(\sqrt{N})$.

\begin{lem}
\label{WN}

Let $N\in\cU$.
Then: %we have the following assertions:
%\begin{enumerate}[\upshape (1)]
%\setlength{\itemsep}{0mm}
%\setlength{\parskip}{0mm}
%\item

\noindent
{\rm (1)}
$B\otimes_{\Q}W_N\cong\M_2(W_N)$.
%\item

\noindent
{\rm (2)}
The prime $\mfq$ is ramified in $kW_N$.
%\item

\noindent
{\rm (3)}
If we let $\mfq'$ be the unique prime of $kW_N$ above $\mfq$,
then $\mfq'$ is of odd degree.
%\end{enumerate}

\end{lem}

\begin{proof}

(1)
This isomorphism holds because any prime divisor of $d$ is ramified in $W_N$.

(2)
The prime number $q$ is ramified in $W_N$ and \textit{is unramified in $k$},
as required.

(3)
By (2) we have $\N(\mfq')=\N(\mfq)$, which is an odd power of $q$.

\end{proof}

Then Proposition \ref{W&q'} is a consequence of the following lemma:

\begin{lem}
\label{Nexist}

There is an integer $N\in\cU$ such that
$kW_N$ does not contain
the Hilbert class field of any imaginary quadratic field.

\end{lem}

\begin{proof}

Assume otherwise i.e.
for any $N\in\cU$, assume that there is an imaginary quadratic field $J_N$
such that $kW_N$ contains the Hilbert class field $H_N$ of $J_N$.
Since $H_N\not\subseteq k$, we have
$k\subsetneq kH_N\subseteq kW_N$.
Then $kH_N=kW_N$ because $[kW_N:k]=2$.
Then
$h_{J_N}=[H_N:J_N]=\frac{1}{2}[H_N:\Q]\leq\frac{1}{2}[kH_N:\Q]
=\frac{1}{2}[kW_N:\Q]=[k:\Q]$.
We see that there are only finitely many such imaginary quadratic fields $J_N$.
We also have $kH_N=kW_N\supseteq W_N$.
Since $\sharp\cU=\infty$, this implies that finitely many number fields
contain infinitely many subfields, which is a contradiction.

\end{proof}

\section{Example}
\label{sec:ex}

We give an example of Theorem \ref{mainthm} (or Theorem \ref{mainthm1})
as follows:

\begin{prop}
\label{ex1}

Suppose $k=\Q(\zeta_{31})$ and $d\in\{6,22\}$. Then:
%\begin{enumerate}[\upshape (1)]
%\setlength{\itemsep}{0mm}
%\setlength{\parskip}{0mm}
%\item

\noindent
{\rm (1)}
$B\otimes_{\Q}k\not\cong\M_2(k)$.
%\item

\noindent
{\rm (2)}
$\sharp M^B(k)=\infty$.
%\item

\noindent
{\rm (3)}
$kW_{-d}$ does not contain
the Hilbert class field of any imaginary quadratic field.
%\item

\noindent
{\rm (4)}
If $p>128$ and $p\not\in\cNn_1(kW_{-d})$, then
$M_0^B(p)(k)=M_0^B(p)(kW_{-d})=\emptyset$.
%\end{enumerate}
\end{prop}

\begin{proof}

For a field $F$ with $\ch F\ne 2$ and two elements $a,b\in F$,
let $\displaystyle\left(\frac{a,b}{F}\right)$
be the quaternion algebra over $F$ defined by
$$\displaystyle\left(\frac{a,b}{F}\right)=F+Fe+Ff+Fef,\ 
e^2=a,\ f^2=b,\ ef=-fe.$$
For a prime number $p$, let $e_p$ (\resp $f_p$, \resp $g_p$)
be the ramification index of $p$ in $k/\Q$
(\resp the degree of the residue field extension above $p$ in $k/\Q$,
\resp the number of primes of $k$ above $p$).

(1)
We have
$(e_2,f_2,g_2)=(1,5,6)$.
Let $v$ be a place of $k$ above $2$.

[Case $d=6$].
We see $\displaystyle B\cong\left(\frac{6,5}{\Q}\right)$ by \cite[\S 3.6 g)]{Shimizu}.
It suffices to prove
$B\otimes_{\Q}k_v\not\cong\M_2(k_v)$.
Since $\Q_2(\sqrt{5})$ is the unramified quadratic extension of $\Q_2$,
the prime number $5$ is not a square in $k_v$
(which is the unramified extension of $\Q_2$ of degree $5$).
We also see that $k_v(\sqrt{5})$ is the unramified quadratic extension of $k_v$.
The $2$-adic valuation
$$\nu:k_v^{\times}\longrightarrow\Z\ ;\ t\longmapsto\ord_2(t)$$
induces an isomorphism
$\nubar:k_v^{\times}/\Norm_{k_v(\sqrt{5})/k_v}(k_v(\sqrt{5})^{\times})\cong\Z/2\Z$.
We have
$6\not\in\Norm_{k_v(\sqrt{5})/k_v}(k_v(\sqrt{5})^{\times})$
because $\nubar(6)=1\ne 0$.
Therefore
$B\otimes_{\Q}k_v\not\cong\M_2(k_v)$, as required.

[Case $d=22$].
We have $\displaystyle B\cong\left(\frac{22,13}{\Q}\right)$ (loc.cit.).
It suffices to prove
$B\otimes_{\Q}k_v\not\cong\M_2(k_v)$.
By the same argument as in [Case $d=6$],
the prime number $13$ is not a square in $k_v$,
and
$22\not\in\Norm_{k_v(\sqrt{13})/k_v}(k_v(\sqrt{13})^{\times})$.

(2)
The curve $M^B$ is defined by $x^2+y^2+3z^2=0$
(\resp $x^2+y^2+11z^2=0$)
in homogeneous coordinates
if $d=6$ (\resp $d=22$)
(\cf \cite[Theorem 1-1]{K}).
For a prime number $p$, we have $M^B(\Q_p)\ne\emptyset$
if and only if $p\ne 3$ (\resp $p\ne 11$)
(\cf \cite[Proof of Lemma 4.4]{A3}).
We see
$(e_3,f_3,g_3)=(1,30,1)$
(\resp $(e_{11},f_{11},g_{11})=(1,30,1)$).
Now we prove $M^B(k_v)\ne\emptyset$
for a place $v$ of $k$ above $3$ (\resp $11$).
Since the order of $\F_{3^{30}}^{\times}$
(\resp $\F_{11^{30}}^{\times}$)
is divisible by $4$,
we have $\sqrt{-1}\in k_v$.
Then $[1,\sqrt{-1},0]\in M^B(k_v)$.
Therefore $M^B(k)\ne\emptyset$.
Since the genus of $M^B$ is $0$, we conclude
$\sharp M^B(k)=\infty$.

(3)
[Case $d=6$].
All the imaginary quadratic subfields of $kW_{-6}$
are $\Q(\sqrt{-6})$ and $\Q(\sqrt{-31})$, %and $\Q(\sqrt{186})$,
whose class numbers are $2$ and $3$ %and $2$ 
respectively.
First assume that $kW_{-6}$ contains the Hilbert class field $H(-6)$
of $\Q(\sqrt{-6})$.
Then $H(-6)=\Q(\sqrt{-6},\sqrt{-31})$,
because we have $[H(-6):\Q]=h_{\Q(\sqrt{-6})}[\Q(\sqrt{-6}):\Q]=4$ and
$\Gal(kW_{-6}/\Q)\cong\Z/2\Z\times\Z/2\Z\times\Z/15\Z$.
But, in the extension 
$\Q(\sqrt{-6},\sqrt{-31})/\Q(\sqrt{-6})$,
the primes of $\Q(\sqrt{-6})$ above $31$
are ramified, which is a contradiction.
Next assume that $kW_{-6}$ contains the Hilbert class field $H(-31)$
of $\Q(\sqrt{-31})$.
Take a prime $\cP$ of $kW_{-6}$ above $31$, and let $\mfp_H$ (\resp $\mfp$)
be the prime of $H(-31)$ (\resp $\Q(\sqrt{-31})$) below $\cP$.
Then $e(\cP/\mfp_H)=15$ because we have
$e(\cP/\mfp)=15$ and $e(\mfp_H/\mfp)=1$,
where $e(\,\cdot\,/\,\cdot\,)$ is the ramification index.
Since $[H(-31):\Q(\sqrt{-31})]=h_{\Q(\sqrt{-31})}=3$,
we have $[kW_{-6}:H(-31)]=10$.
This contradicts $e(\cP/\mfp_H)=15$.

[Case $d=22$].
All the imaginary quadratic subfields of $kW_{-22}$
are $\Q(\sqrt{-22})$ and $\Q(\sqrt{-31})$,
whose class numbers are $2$ and $3$
respectively.
Then we are done by the same argument as in [Case $d=6$].

(4)
The least $\N(\mfq')$ for primes $\mfq'$ of $kW_{-d}$, of odd degree,
whose residual characteristic $q$
satisfies $B\otimes_{\Q}\Q(\sqrt{-q})\not\cong\M_2(\Q(\sqrt{-q}))$,
is $2^5=32$.
Then the assertion follows from Theorem \ref{mainthm0} (1)
and Lemma \ref{WN} (1).

\end{proof}

\def\bibname{References}

(Keisuke Arai)
Department of Mathematics, School of Engineering,
Tokyo Denki University,
5 Senju Asahi-cho, Adachi-ku, Tokyo 120-8551, Japan

\textit{E-mail address}: \texttt{araik@mail.dendai.ac.jp}


\begin{thebibliography}{}

%\bibitem{A1} \textit{K.\ Arai},
%On the Galois images associated to QM-abelian surfaces,
%Proceedings of the Symposium on Algebraic
%Number Theory and Related Topics, 165--187, 
%RIMS K\^{o}ky\^{u}roku Bessatsu, {\bf B4}, Res. Inst. Math. Sci. (RIMS), Kyoto, 2007.

%\bibitem{A2} \textit{K.\ Arai},
%Galois images and modular curves,
%Proceedings of the Symposium on Algebraic
%Number Theory and Related Topics, 145--161, 
%RIMS K\^{o}ky\^{u}roku Bessatsu, {\bf B32}, Res. Inst. Math. Sci. (RIMS), Kyoto, 2012.

\bibitem{A3} \textit{K.\ Arai},
On the Rasmussen-Tamagawa conjecture for QM-abelian surfaces,
to appear in RIMS K\^{o}ky\^{u}roku Bessatsu,
available at the web page
(http://arxiv.org/pdf/1211.0599v2.pdf).

\bibitem{A4} \textit{K.\ Arai},
Algebraic points on Shimura curves of $\Gamma_0(p)$-type (II),
preprint,
available at the web page
(http://arxiv.org/pdf/1205.3596v2.pdf).

\bibitem{A5} \textit{K.\ Arai},
An effective bound of $p$ for algebraic points on Shimura curves of $\Gamma_0(p)$-type,
preprint,
available at the web page
(http://arxiv.org/pdf/1211.0129v2.pdf).

\bibitem{AM} \textit{K.\ Arai, F.\ Momose},
Algebraic points on Shimura curves of $\Gamma_0(p)$-type,
J. Reine Angew. Math.,
published online (ahead of print).
%available at the web page
%(http://arxiv.org/pdf/1202.4841v2.pdf).

\bibitem{AM2} \textit{K.\ Arai, F.\ Momose},
Errata to Algebraic points on Shimura curves of $\Gamma_0(p)$-type,
J. Reine Angew. Math.,
published online (ahead of print).

%\bibitem{Ba} \textit{A.\ Baker},
%Linear forms in the logarithms of algebraic numbers, %I,II,III
%Mathematika {\bf 13} (1966), 204--216.
%; ibid. 14 (1967), 102--107; ibid. 14 (1967) 220--228.

%\bibitem{Be} \textit{A.\ Besser},
%CM cycles over Shimura curves,
%J. Algebraic Geom. {\bf 4} (1995), no. 4, 659--691.

\bibitem{Bu} \textit{K.\ Buzzard}, 
Integral models of certain Shimura curves, 
Duke Math. J. {\bf 87} (1997), no. 3, 591--612.

%\bibitem{DD} \textit{P.\ D\`{e}bes, J.-C.\ Douai},
%Algebraic covers: field of moduli versus field of definition,
%Ann. Sci. Ecole Norm. Sup. (4) {\bf 30} (1997), no. 3, 303--338.

%\bibitem{DR} \textit{P.\ Deligne, M.\ Rapoport},
%Les sch\'{e}mas de modules de courbes elliptiques,
%Modular functions of one variable, II, 
%143--316. Lecture Notes in Math., Vol. {\bf 349}, Springer, Berlin, 1973.

%\bibitem{Fa} \textit{G.\ Faltings},
%Endlichkeitss\"{a}tze f\"{u}r abelsche Variet\"{a}ten \"{u}ber Zahlk\"{o}rpern,
%Invent. Math. {\bf 73} (1983), no. 3, 349--366.

\bibitem{J} \textit{B.\ Jordan},
Points on Shimura curves rational over number fields,
J. Reine Angew. Math. {\bf 371} (1986), 92--114. 

\bibitem{K} \textit{A.\ Kurihara},
On some examples of equations defining Shimura curves
and the Mumford uniformization,
J. Fac. Sci. Univ. Tokyo Sect. IA Math. {\bf 25} (1979), no. 3, 277--300. 

%\bibitem[Ma1]{Maz1} B.\ Mazur,
%{\em Modular curves and the Eisenstein ideal}, 
%I.H.E.S. Publ. Math. No. 47 (1977), 33-186. 

%\bibitem[Maz2]{Maz2} B.\ Mazur,
%{\em Rational points on modular curves},
%Modular functions of one variable V, 
%Lecture Notes in Math., Vol. 601, Springer, Berlin (1977), 107-148.

\bibitem{Ma} \textit{B.\ Mazur},
Rational isogenies of prime degree (with an appendix by D. Goldfeld), 
Invent. Math. {\bf 44} (1978), no. 2, 129--162.

\bibitem{Mo} \textit{F.\ Momose},
Isogenies of prime degree over number fields,
Compositio Math. {\bf 97} (1995), no. 3, 329--348.

\bibitem{Oh} \textit{M.\ Ohta}, 
On $l$-adic representations of Galois groups obtained from certain 
two-dimensional abelian varieties, 
J. Fac. Sci. Univ. Tokyo Sect. IA Math. {\bf 21} (1974), 299--308.

%\bibitem{O1} \textit{Y.\ Ozeki},
%Non-existence of certain Galois representations with a uniform tame inertia weight,
%Int. Math. Res. Not. {\bf 2011} (2011), no. 11, 2377--2395. 

%\bibitem{O2} \textit{Y.\ Ozeki},
%Non-existence of certain CM abelian varieties with prime power torsion,
%preprint, available at the web page
%(http://arxiv.org/pdf/1112.3097.pdf).

%\bibitem{RT} \textit{C.\ Rasmussen, A.\ Tamagawa},
%A finiteness conjecture on abelian varieties with constrained prime power torsion,
%Math. Res. Lett. {\bf 15} (2008), no. 6, 1223--1231. 

%\bibitem[Si]{Si} J.\ Silverman,
%{\em Advanced topics in the arithmetic of elliptic curves},
%Graduate Texts in Mathematics, 151. Springer-Verlag, New York, 1994.

\bibitem{Shimizu} \textit{H.\ Shimizu}, 
Hokei kans\={u}. I--III (Japanese) [Automorphic functions. I--III] 
Second edition.
Iwanami Shoten Kiso S\={u}gaku [Iwanami Lectures on Fundamental Mathematics], 
8. Dais\={u} [Algebra], vii. Iwanami Shoten,
Tokyo, 1984. 

%\bibitem{Sh0} \textit{G.\ Shimura},
%Introduction to the arithmetic theory of automorphic functions,
%Kano Memorial Lectures, No. 1. Publications of the Mathematical Society of Japan, No. 11.
%Iwanami Shoten, Publishers, Tokyo; Princeton University Press, Princeton, N.J., 1971.

\bibitem{Sh} \textit{G.\ Shimura},
On the real points of an arithmetic quotient of a bounded symmetric domain,
Math. Ann. {\bf 215} (1975), 135--164.

%\bibitem{We} \textit{A.\ Weil},
%Basic number theory,
%Reprint of the second (1973) edition. Classics in Mathematics. Springer-Verlag, Berlin, 1995.

%\bibitem{Z} \textit{Y.\ G.\ Zarhin},
%A finiteness theorem for unpolarized abelian varieties over number fields
%with prescribed places of bad reduction,
%Invent. Math. {\bf 79} (1985), no. 2, 309--321.

\end{thebibliography}
\end{document}